\documentclass[11pt,reqno,a4paper]{amsart} 
\usepackage[applemac]{inputenc}
\usepackage{amsmath}
\usepackage{amssymb}
\usepackage{euscript}
\usepackage{mathabx}
\usepackage{mathrsfs}
\usepackage{wasysym}
\usepackage[all]{xy}
\usepackage{graphicx}
\usepackage{color}
\usepackage{multirow}
\usepackage{pifont}
\usepackage[normalem]{ulem}
\usepackage[table]{xcolor}
\usepackage[colorlinks=true,linktocpage=true,pagebackref=false, urlcolor=black, citecolor=black,linkcolor=black]{hyperref}
\usepackage{textcomp}

\usepackage[backrefs,lite]{amsrefs}  
\usepackage{enumitem}
\usepackage{comment}

\usepackage{academicons}			
\newcommand{\orcid}[1]{\href{https://orcid.org/#1}{\textcolor[HTML]{A6CE39}{\aiicon{googlescholar}}}}

\newtheorem{thm}{Theorem}[section]
\newtheorem*{thmA}{Theorem A}
\newtheorem*{thmB}{Theorem B}

\newtheorem{lem}[thm]{Lemma}

\newtheorem{cor}[thm]{Corollary}

\theoremstyle{definition}
\newtheorem{defn}[thm]{Definition}

\theoremstyle{remark}

\newtheorem{example}[thm]{Example}

\numberwithin{equation}{section}

 \newcommand{\N}{{\mathbb N}}
 \newcommand{\R}{{\mathbb R}}

\newcommand{\sph}{{\mathbb S}}

\numberwithin{equation}{section}

\begin{document}
\title[On the degree of global smoothings for subanalytic sets]{On the degree of global smoothings for subanalytic sets}

\author{Enrico Savi}
\address{Dipartimento di Matematica, Via Sommarive, 14, Università di Trento, 38123 Povo (ITALY)}
\email{enrico.savi@unitn.it}
\thanks{The author is supported by GNSAGA of INDAM} 



\begin{abstract}
In \cite{bier-par} Bierstone and Parusi\'nski proved the existence of global smoothings for closed subanalytic sets, both in an embedded and a non-embedded sense. In particular, in the non-embedded desingularization procedure the authors construct smoothings of (generically) even degree, indeed it is well-known the existence of subanalytic sets which do not admit non-embedded smoothings of (generically) odd degree. In this paper we introduce a natural topological notion of nonbounding equator for subanalytic sets and we prove a criterion to determine whether a closed subanalytic set $X$ only admits global smoothings of even degree along the nonbounding equator. More in detail, we prove that if $X$ has a nonbounding equator $Y$ then every smoothing of $X$ which is a covering on a connected neighborhood $W$ of $Y$ has even degree over $W$.
\end{abstract}

\date{\today}

\maketitle

\section{Introduction}

In \cite{bier-par} Bierstone and Parusi\'nski proved the following two remarkable global smoothing results for closed subanalytic sets. The term `analytic' means `real analytic'. \emph{Let $V$ be an analytic manifold of dimension $n$, and let $X$ be a closed subanalytic subset of $V$ of dimension $k$.}


\begin{thmA}[{\cite[Theorem\,1.1]{bier-par}} Non-embedded global smoothing]
	There exist an analytic mani\-fold $X'$ of pure dimension $k$, a proper analytic mapping $\varphi:X'\to V$, and a smooth open subanalytic subset $U$ of $X$ such that:
	\begin{enumerate}[label=\emph{(\arabic*)}, ref=\emph{(\arabic*)}]
		\item $\varphi(X')\subset X$.
		\item $\dim (X\setminus U)<k$ and $\varphi^{-1}(X\setminus U)$ is a simple normal crossings hypersurface $B'$ of $X'$.
		\item For each connected component $W$ of $U$, $\varphi^{-1}(W)$ is a finite union of subsets open and closed in $\varphi^{-1}(U)$, each mapped isomorphically onto $W$ by $\varphi$.
	\end{enumerate}
\end{thmA}


\begin{thmB}[{\cite[Theorem\,1.2]{bier-par}} Embedded global smoothing]
	There exist an analytic manifold $V'$, a smooth closed analytic subset $X'\subset V'$ of dimension $k$, a simple normal crossings hypersurface $B'\subset V'$ transverse to $X'$, and a proper analytic mapping $\varphi:V'\to V$ such that:
	\begin{enumerate}[label=\emph{(\arabic*)}, ref=\emph{(\arabic*)}]
		\item $\dim(\varphi(B'))<k$.
		\item The restriction $\varphi|_{V'\setminus B'}$ is finite-to-one and of constant rank $n$;
		\item $\varphi$ induces an isomorphism from a union of connected components of $X'\setminus B'$ to a smooth open subanalytic subset $U\subset X$ such that $\dim(X\setminus U)<k$.
	\end{enumerate}
\end{thmB}


Although the previous results are global, the techniques involved in their proves are local nature. Indeed, in \cite[\S 2.3]{bier-par} the authors provide a partition of the analytic manifold $V$ into a countable number of semianalytic cells in general position with respect to $X$ and then they develop explicit desingularization techniques for these cells with respect to the global behaviour of $V$. More in detail, in \cite[\S 2.2]{bier-par} the authors develop a desingularizing procedure for a semianalytic $n$-cell $C$ of $V$ by explicitly finding an analytic subset $Z_C$ of $V\times\R^m$, for some $m\in\N$ depending on the number of inequalities defining $C$, a map $\varphi_C:Z_C\to C$ and an open semianalytic subset $U_C$ of $C$ such that $\varphi_{C}^{-1}(U_C)$ is a $2^m$ covering of $U_C$ and $\dim(C\setminus U_C)<k$. Then the authors apply desingularitazion techniques in the sense of \cite{bier-mil:hironaka} to $Z_C$ finding a smoothing of the cell $C$. Thus, we see that the smoothing $\varphi_C$ of a single cell $C$ of $V$ is even-to-one over $U_C$. Since the global maps $\varphi$ in Theorem A and Theorem B are constructed in terms of the local maps $\varphi_C$, we deduce that $\varphi$ is even-to-one over each open set $U_C$, hence, in particular, $\varphi$ is even-to-one over each intersection $U_C\cap X$.

Let us give a definition.


\begin{defn}\label{def:gss}
	Let $X'$, $\varphi$, $U$ and $W$ be as in the previous \emph{Theorem A}, that is, $X'$ is an analytic manifold of pure dimension $k$, $\varphi:X'\to V$ is a proper analytic mapping, $U$ is an open subanalytic subset of the smooth part of $X$ of dimension $k$ and $W$ is a connected component of $U$ such that $\varphi(X')\subset X$, $\dim(X\setminus U)<k$, $\varphi^{-1}(X\setminus U)$ is a simple normal crossings hypersuface of $X'$ and $\varphi^{-1}(W)$ is a finite union of subsets open and closed in $\varphi^{-1}(U)$, each mapped isomorphically onto $W$ by $\varphi$.  We call the triple $\Gamma:=(X',\varphi,U)$ \emph{global smoothing section of $X\subset V$} and the finite positive number of subsets open and closed in $\varphi^{-1}(U)$, each mapped isomorphically onto $W$ by $\varphi$, as the \emph{degree of $\Gamma$ over $W$}.
\end{defn}


Theorem A asserts that global smoothing sections of $X\subset V$ always exist.

The aim of this note is to give in Theorem \ref{main} a criterion for the evenness of the degree of global smoothing sections of $X\subset V$. This criterion shows that there may exist obstructions of a global nature to the possibility of constructing global smoothing for $X$ that are odd-to-one and, in particular, one-to-one as in the case of Hironaka's resolution of singularities. See Examples \ref{ex:halfline} and \ref{ex:2} below. 

As we said, our evenness criterion concerns Non-embedded global smoothing of Theorem A. In the Embedded setting, the situation may be different. Indeed, by \cite[Remark\,2.6]{bier-par}, if $V=\R^n$ and $X$ is a closed semialgebraic subset of $\R^n$, then Theorem B can be strengthened by requiring the mapping in (2) to be injective.    


\section{The evenness criterion, consequences and examples}

By Whitney's embedding theorem we can assume that the analytic manifold $V$ coincides with $\R^n$. Let $X$ be a subanalytic subset of $\R^n$ and let $k\in\N$. Recall that a point $x\in X$ is \emph{smooth of dimension $k$} if there exists an open neighborhood $N$ of $x$ in $\R^n$ such that $X\cap N$ is an analytic submanifold of $\R^n$ of dimension $k$, see \cite[Definition\,3.3]{bier-mil:sub}. The set of all smooth points of $X$ of dimension $k$ is an open subanalytic subset of $X$, see \cite[\S 7]{bier-mil:sub}, hence it is an analytic submanifold of $\R^n$ of pure dimension $k$.

Let us introduce the concept of nonbounding equator for subanalytic sets.


\begin{defn} \label{def:ne}
	Let $X$ be a closed subanalytic subset of $\R^n$ of dimension $k$, let $Y$ be a subset of the smooth part of $X$ of dimension $k$. We say that $Y$ is a \emph{nonbounding equator of $X$} if it satisfies the following properties:
	\begin{enumerate}[label={(\roman*)},ref=(\roman*)]
		\item\label{def:ne1} $Y$ is a compact $\mathscr{C}^\infty$ submanifold of $\R^n$ of dimension $k-1$.
		\item\label{def:ne2} $Y$ does not bound, that is, it is not the boundary of a compact $\mathscr{C}^\infty$ manifold with boundary.
		\item\label{def:ne3} $Y$ has a collar in $X$, that is, there exists a $\mathscr{C}^\infty$ map $\psi:Y\times(-1,1)\to X$ such that the image $T:=\psi(Y\times(-1,1))$ of $\psi$ is an open neighborhood of $Y$ in the set of smooth points of $X$ of dimension $k$, the restriction $\psi:Y\times(-1,1)\to T$ is a $\mathscr{C}^\infty$ diffeomorphism and $\psi(Y\times\{0\})=Y$.
		\item\label{def:ne4} There exists a relatively compact open subset $K$ of $X$ such that $\partial K:=\overline{K}\setminus K=Y$ and $K\cap T=\psi(Y\times(-1,0))$. Here $\overline{K}$ denotes the closure of $K$ in $X$.
	\end{enumerate}
	If such a $Y$ exists, we say that \emph{$X$ has a nonbounding equator}.
\end{defn}


The next lemma gives an alternative description of the notion of nonbounding equator. We keep the notations of Definition \ref{def:ne}.


\begin{lem}\label{lem}
	The set $Y$ is a nonbounding equator of $X$ if and only if there exists a continuous function $h:X\to\R$ with the following properties:
	\begin{enumerate}[label=\emph{(\roman*)},ref=(\roman*)]
		\item\label{lem1} There exist an open neighborhood $Z$ of $Y$ in the set of smooth points of $X$ of dimension $k$ and $\epsilon>0$ such that the restriction $h':=h|_Z:Z\to\R$ is a $\mathscr{C}^\infty$ function, $h^{-1}([-\epsilon,\epsilon])$ is a compact neighborhood of $Y$ in $Z$ containing no critical points of $h'$ and $h^{-1}(0)=Y$.
		\item\label{lem2} $Y$ does not bound.
		\item\label{lem3} The subset $h^{-1}((-\infty,0])$ of $X$ is compact. 
	\end{enumerate}   
\end{lem}

\begin{proof}
	Let $X$, $k$, $Y$, $\psi:Y\times(-1,1)\to X$ and $K$ be as in Definition \ref{def:ne} and let $\pi:Y\times(-1,1)\to (-1,1)$ be the projection onto the second factor. Let us prove that Lemma \ref{lem}\ref{lem1}-\ref{lem3} are satisfied. Define $Z:=\psi(Y\times (-1/2,1/2))$ and $h':Z\to\R$ as $h'(x):=(\pi\circ\psi)^{-1}(x)$. Then extend $h'$ to the whole $X$ as follows: define $h:X\to\R$ as $h(x):=-1/2$ if $x\in K\setminus Z$, $h(x):=h'(x)$ if $x\in Z$ and $h(x):=1/2$ otherwise. Fix $\epsilon:=1/4$. Observe that $h|_{Z}=(\pi\circ\psi^{-1})|_{Z}$, thus $h|_{Z}$ has no critical points, $h^{-1}([-1/4,1/4])=\psi(Y\times[-1/4,1/4])$, which is compact and contains $Y$, and $h^{-1}((-\infty,0])=K\cup Y=\overline{K}$.
	
	On the other hand, assume that $X$, $Y$, $Z$ and $h$ satisfy Lemma \ref{lem}\ref{lem1}-\ref{lem3}. By Lemma \ref{lem}\ref{lem1} and \cite[Corollary\,2.3, p.\,154]{hirsh:difftop}, $h|_{h^{-1}([-\varepsilon,\varepsilon])}$ induces the existence of a collar of $Y$ in $X$, as in Definition \ref{def:ne}\ref{def:ne3}. Moreover, by Lemma \ref{lem}\ref{lem1}\ref{lem3}, $K:=h^{-1}((-\infty,0))$ satisfies Definition \ref{def:ne}\ref{def:ne4}.
\end{proof}

Our evenness criterion reads as follows.


\begin{thm}\label{main}
	Let $X$ be a closed subanalytic subset of $\R^n$, let $\Gamma:=(X',\varphi,U)$ be a global smoothing section of $X\subset\R^n$ and let $W$ be a connected component of $U$. If $W$ contains a nonbounding equator of $X$, then the degree of $\Gamma$ over $W$ is even.
\end{thm}

\begin{proof}
	Let $Y$ be a nonbounding equator of $X$. By Definition \ref{def:ne}, there is an open neighborhood $T$ of $Y$ in the set of smooth points of $X$ of dimension $k$, a diffeomorphism $\psi:Y\times (-1,1)\to T$ such that $\psi(Y\times\{0\})=Y$ and a relatively compact open subset $K$ of $X$ such that $\partial K=Y$ and $K\cap T=\psi(Y\times (-1,0))$. Since $Y\subset W$ and $W$ is open, we may suppose $T\subset W$. Since $\Gamma$ is a global smoothing section, $\varphi^{-1}(W)$ consists of a finite disjoint union of open and closed subsets of $\varphi^{-1}(U)$, each mapped isomorphically onto $W$. Hence, each connected component of $\varphi^{-1}(W)$ contains a copy of $Y$ and a copy of the collar $T$ of $Y$ in $X$. By Definition \ref{def:gss}, the map $\varphi$ is proper, hence $\varphi^{-1}(\overline{K})$ is a compact subset of $X'$. Moreover, since $\partial K=Y$, $K\cap T=\psi(Y\times (-1,0))$ and $\varphi$ is a diffeomorphism when restricted to each connected component of $\varphi^{-1}(W)$, we have that $\varphi^{-1}(\overline{K})$ is a manifold with boundary whose boundary is the disjoint union of $d$ copies of $Y$, where $d$ denotes the degree of $\Gamma$ over $W$. Since $Y$ is nonbounding, we deduce that $d$ is even since the Stiefel-Whitney numbers of $\bigsqcup_{1}^d Y$ must be all zero, see \cite[Theorem\,4.9, p.\,52]{mil-sta}.
\end{proof}

As a consequence, the nonexistence of nonbounding equators of $X$ is a necessary condition to have global one-to-one smoothings similar to Hironaka's resolution of singularities.


\begin{cor}
	Let $X$ be a closed subanalytic subset of $\R^n$ of dimension $k$ and let $W$ be a connected component of the smooth part of $X$ of dimension $k$. If the degree of a global smoothing section of $X\subset\R^n$ over $W$ is $1$, then $W$ does not have any nonbounding equator in $X$.
\end{cor}


Here we present some examples of semialgebraic sets concerning our Theorem \ref{main}.


\begin{example}\label{ex:halfline}
	Let $X:=\R_{\geq 0}:=\{x\in\R\,|\,x\geq 0\}$. There is a global smoothing section of the whole smooth part of $X$, that is $\Gamma:=(X',\varphi,U)$ with  $U:=\R_{>0}=\{x\in\R\,|\,x> 0\}$, $X':=\{(x,y)\in\R^2\,|\,x=y^2\}$ and $\varphi:X'\to X$ defined as the projection onto the first factor. According to Theorem \ref{main}, the degree of the above smoothing section over the whole smooth part of $X$ is $2$, so it is even. Actually, Theorem \ref{main} says something more, indeed any global smoothing section $\Gamma:=(X',\varphi,U)$ of $X$, with $U$ any open subset of the smooth part of $X$, has even degree over any connected component of $U$. Indeed, since $U$ is an open subset of $\R_{>0}$, every connected component of $U$ contains a nonbounding equator $Y$ of $X$ consisting of a singleton $\{p\}$, with $K:=[0,p)$ and the collar $(p-\varepsilon,p+\varepsilon)\subset U$ of $p$ in $W$, for $\epsilon>0$ sufficiently small. 
\end{example}


\begin{example}\label{ex:2}
	Let $M$ be a connected compact $\mathscr{C}^\infty$ manifold of dimension $k-1$, which does not bound (so $k-1\geq2$): for instance, the real projective plane $\mathbb{P}^2(\R)$. By the Nash-Tognoli theorem, see \cite{nash,tognoli:nash-conj}, we can assume $M$ is a compact nonsingular real algebraic subset of some $\R^n$.
	\begin{enumerate}
		\item Consider the standard circumference $\sph^1:=\{(a,b)\in\R^2:a^2+b^2=1\}$, the compact nonsingular real algebraic set $X':=M\times\sph^1\subset\R^{n+2}$, and the polynomial maps $\pi_1:X'\to\R^{n+2}$ and $\pi_2:\R^{n+2}\to\R^{n+2}$ defined as follows:
		\[
		\pi_1(x,a,b):=(bx,a,b) \quad \text{and}
		\quad
		\pi_2(x,a,b):=(x,a,b^2),
		\]
		where $x=(x_1,x_2,\ldots,x_n)$. The set $\pi_1(X')$ is equal to $X'$ with $M\times\{(-1,0)\}$ crushed to the point $p:=(0,\ldots,0,-1,0)$ and $M\times\{(1,0)\}$ crushed to the point $q:=(0,\ldots,0,1,0)$. The set $X:=\pi_2(\pi_1(X'))$ is a semialgebraic subset of $\R^{n+2}$ homeomorphic to the suspension of $M$. Define $X'_\pm:=X'\cap\{\pm b>0\}$ and the polynomial map $\varphi:X'\to\R^{n+2}$ by $\varphi(x,a,b):=\pi_2(\pi_1(x,a,b))$. Observe that $\varphi(X')=X$, $\varphi^{-1}(p)=M\times\{(-1,0)\}$, $\varphi^{-1}(q)=M\times\{(1,0)\}$, and the restriction of $\varphi$ from $X'_\pm$ to $U:=X\setminus\{p,q\}$, namely to the whole smooth part of $X$, is a Nash diffeomorphism between connected Nash manifolds. For more details about Nash functions and Nash manifolds we refer to \cite[\S 8]{BCR}. The triple $\Gamma:=(X',\varphi,U)$ is a global smoothing section of $X\subset\R^{n+2}$ and $\varphi(M\times\{(0,1)\})$ is a nonbounding equator of $X$. The degree of $\Gamma$ over $U$ is two, in accordance with our Theorem \ref{main}.
		
		\item Let $X':=M\times[-1,1]\subset\R^{n+1}$, let $\phi:X'\to\R^{n+1}$ be the polynomial map
		\[
		\phi(x,a):=(x(1-a^2),a)
		\]
		and let $X$ be the semialgebraic subset $\phi(X')$ of $\R^{n+1}$. Observe that $X$ is homeomorphic to the suspension of $M$, $\phi^{-1}(z_\pm)=M\times\{\pm1\}$, where $z_\pm:=(0,\ldots,0,\pm1)$, the restriction of $\phi$ from $X'\setminus(M\times\{-1,1\})=M\times(-1,1)$ to $U:=X\setminus\{z_-,z_+\}$ is a Nash diffeomorphism between connected Nash manifolds (so $\phi$ has degree one over $U$), and $\phi(M\times\{0\})$ is a nonbounding equator of $X$. However, the triple $(X',\phi,U)$ is not a global smoothing section of $X\subset\R^{n+1}$ since $X'$ is not an analytic manifold: it has the nonempty boundary $M\times\{-1,1\}$.
		
		Nevertheless, the previous construction arises as an explicit case of Theorem B. Let $V:=\R^{n+1}$. By \cite[Corollary\,2.5.14, p.\,50]{akbking:top} we may assume in addition that $M$ is projectively closed, that is $M$ is the zero set $\mathcal{Z}_{\R^n}(p)$ in $\R^n$ of some overt polynomial $p\in\R[x_1,\dots,x_n]$.  Write $p$ as follows: $p=\sum_{i=0}^dp_i$, where $p_i$ is an homogeneous polynomial of degree $i$.  Recall that $\mathcal{Z}_{\R^n}(p_d)=\{0\}$ as $p$ is overt. Thus, if $\varphi:\R^{n+1}\to\R^{n+1}$ is the polynomial map $(x,a)\mapsto((1-a^2)x,a)$, $Z':=M\times\R$, $Z:=\varphi(Z')$ and $q(x,a)\in\R[x_1,\ldots,x_n,a]$ is the polynomial $q(x,a):=\sum_{i=0}^d(1-a^2)^{d-i}p_i(x)$, then $Z'=\mathcal{Z}_{\R^{n+1}}(p)$ and $q(\varphi(x,a))=(1-a^2)^d p(x)=0$ for all $(x,a)\in Z'$. It follows that
		\[
		Z=\mathcal{Z}_{\R^{n+1}}(q).
		\]
		This proves that $Z$ is algebraic and irreducible, so $Z$ is the Zariski closure of $X$ in $\R^{n+1}$. Thus, 
		we deduce that $X$, $Y:=\{z_{-},z_{+}\}$, $Z$, $U$, $\varphi$, $Z'$ and $B':=\varphi^{-1}(Y)$ constitute an explicit embedded global smoothing as in \cite[Remark\,2.6]{bier-par}.
	\end{enumerate}
\end{example}

\section*{Acknowledgments}

I would like to thank Riccardo Ghiloni for suggesting to investigate the topics of this article and for valuable discussions during the drafting process.

\end{document}